\documentclass[12pt]{article}

\usepackage{amstext,amssymb,amsmath,amsbsy}
\usepackage{hyperref}
\usepackage{amscd}
\usepackage{amsfonts}
\usepackage{indentfirst}
\usepackage{verbatim}
\usepackage{amsmath}
\usepackage{amsthm}
\usepackage{enumerate}
\usepackage{graphicx}
\usepackage{color, soul}
\usepackage[OT1]{fontenc}
\usepackage[latin1]{inputenc}
\usepackage[english]{babel}
\usepackage{amssymb}
\usepackage{subfig}
\usepackage{algorithm}
\usepackage{algorithmic}

\usepackage{mathtools}
\DeclarePairedDelimiter\ceil{\lceil}{\rceil}

\newcommand{\R}{\mathbb{R}}

\newcommand{\ds}{\displaystyle}

\newcommand{\x}{{\bf x}}
\newcommand{\p}{{\bf p}}

\newtheorem{Theorem}{Theorem}[section]
\newtheorem{Lemma}{Lemma}[section]

\newtheorem{Corollary}{Corollary}[section]
\newtheorem{remark}{Remark}[section]

\newtheorem*{Assumption*}{Assumption}

\newtheorem{problem}{Problem}[section]
\newtheorem*{problem*}{Problem}
\numberwithin{equation}{section}

\begin{document}

\title{The Carleman-based contraction principle to  reconstruct  the potential of nonlinear hyperbolic equations }

\author{Dinh-Liem Nguyen\thanks{Department of Mathematics, Kansas State University, Manhattan, KS 66506, USA, dlnguyen@ksu.edu, trungt@ksu.edu.}  \and  
 Loc H. Nguyen\thanks{Department of Mathematics and Statistics, University of North Carolina at
Charlotte, Charlotte, NC 28223, USA, loc.nguyen@uncc.edu.}
\and Trung Truong\footnotemark[1]
}

\date{}
\maketitle
\begin{abstract} 
 We develop  an efficient and  convergent numerical method for solving the inverse problem of determining  the potential of  nonlinear hyperbolic equations from lateral Cauchy  data.  In our numerical method we  construct a sequence of linear Cauchy problems whose corresponding solutions  converge to a function that can be used to efficiently compute an approximate solution  to the inverse problem of interest.   The convergence analysis   is established by combining  the contraction principle and Carleman estimates.  We numerically solve the  linear Cauchy  problems using a quasi-reversibility method.   Numerical examples are presented to illustrate the efficiency of the method. 

\end{abstract}



\section{Introduction}
Let $d \geq 2$ be the spatial dimension. 
Let $T$ be a positive number representing the final time.
Let $F: \R^d \times \R \times \R \times \R^d \to \R$ and $p:\R^d \to \R$ 
be smooth functions. 
 Consider the following nonlinear hyperbolic equation
\begin{equation}
	\left\{
		\begin{array}{rcll}
			u_{tt}(\x, t) &=& \Delta u(\x, t) + c(\x) F(\x, u, u_t, \nabla u) &(\x, t) \in \R^d \times (0, T),\\
			u(\x, 0) &=& p(\x) &\x \in \R^d,\\
			u_t(\x, 0) &=& 0 &\x \in \R^d.
		\end{array}
	\right.
	\label{waveeqn}
\end{equation}
Here,  $c: \R^d \to \R$ is the unknown potential of the medium.
The unique solvability of problem \eqref{waveeqn} is out of the scope of this paper.
We consider it as an assumption. 
For the completeness, we provide here a simple set of conditions on $c$, p and $F$ that guarantee the well-posedness of \eqref{waveeqn}. Consider the case when $p$ and $c$ are smooth and have compact support.
 If $F$ does not depend on the first derivatives of $u$,   $|F(\x, u)| \leq c_1 |u| + c_2$ for all $u \in H^2(\Omega \times (0,T))$ and $\x \in \R^d$  for some positive constants $c_1$ and $c_2$, then the unique solvability of problem \eqref{waveeqn} was established in~\cite{Lions:sp1972}. 
We also refer the reader to the approach based on Garlerkin approximations and energy estimates 
 in \cite[Section 7.2]{Evans:PDEs2010} and \cite{Ladyzhenskaya:sv1985} for the unique solvability of \eqref{waveeqn} for the 
case when \eqref{waveeqn} is linear.  Also in the linear case, we refer the reader to \cite[Chapter 10]{Brezis:Springer2011} 
for some results on existence, uniqueness and regularity  of the wave equation. In this paper, we only focus on 
our approach for solving the following  inverse problem.

\begin{problem} Let $\Omega$ be an open and bounded domain of $\R^d$ with smooth boundary.
	Assume that $F(\x, p, 0, \nabla p) \not = 0$ for all $\x \in \overline \Omega$ and that $c(\x)$ is compactly supported in  $\Omega.$
	%Let $T$ be a positive number.
	Given the lateral Cauchy data
	\begin{equation}
		f(\x, t) = u(\x, t) \quad 
		\mbox{and}
		\quad
		g(\x, t) = \partial_{\nu} u(\x, t)
		\label{data}
	\end{equation}  for all $(\x, t) \in \partial \Omega \times [0, T],$
	determine the potential $c(\x)$ in $\Omega$. 
	\label{p}
\end{problem}

This  problem belongs to the class of inverse problems for nonlinear partial differential equations (PDEs) which 
has received an increasing attention in the recent years. Since many materials in
real-world applications obey nonlinear laws, inverse problems for nonlinear PDEs  arise from the  question of how to
  image these materials and/or  determine their physical parameters.   
However,  results on inverse problems for nonlinear models are  quite limited when comparing with those 
on inverse problems for  linear models.
We refer to~\cite{Harju2015} 
for a result on a linearized inverse scattering problem for nonlinear Schr\"odinger equations. A more recent work  
on the numerical solution to a (linear) inverse source problem for nonlinear parabolic equations can be found in~\cite{LeNguyen:2020}.
Due to the presence of the nonlinearity, the uniqueness of Problem \ref{p} is still open. 
We consider the uniqueness result as an assumption.
 We refer to \cite{Beilina2015, BukhgeimKlibanov:smd1981, KlibanovMalinsky:ip1991, Rakesh2020} and references therein  for the uniqueness and stability results and reconstruction methods for coefficient inverse problems for linear hyperbolic equations.
  We also draw the reader's attention to \cite{Runde2008, Serov2008, Isako2011, Kuryl2018, Wang2019, Claud2020}
and references therein for some uniqueness and stability results for inverse problems for several types of nonlinear PDEs. 

The conventional approaches to solve nonlinear inverse problems are based on optimization.
That means, one introduces some mismatch functionals and then finds their minimizers. 
The found minimizers are set to be the solutions to the nonlinear inverse problems.
However, in reality, these functionals might have multiple local minima. 
Finding the minimizers, that are close to the desired solutions, requests good initial guesses.
Even when  good initial guesses are given, the local convergence does not always hold unless some additional conditions are imposed.
We refer the reader to \cite{Hankeetal:nm1995} for a condition that guarantees the local convergence of the optimization method using Landweber iteration.
Unlike this,
in this paper, we propose a new method to solve the nonlinear coefficient inverse problem for nonlinear PDEs, formulated in Problem \ref{p}, without requesting a good initial guess. 
Our approach is to employ a suitable Carleman weight function to construct a contraction map.
The fixed point of this map directly provides an approximate solution to Problem \ref{p}. The contractional phenomenon is proved by using a Carleman estimate.
Besides the proposed method,
there is a general framework to solve nonlinear inverse problems without requesting good initial guesses. It is named as convexification. See \cite{KlibanovNik:ra2017, KhoaKlibanovLoc:SIAMImaging2020, Klibanov:sjma1997, Klibanov:nw1997, Klibanov:ip2015,  KlibanovKolesov:cma2019, KlibanovLiZhang:ip2019, KlibanovLeNguyenIPI2021,  KlibanovLiZhang:SIAM2019, LeNguyen:preprint2021}
for several versions of the convexification method. 
These works have
been developed since the convexification method was originally introduced in \cite{KlibanovIoussoupova:SMA1995}. 
Especially, the convexification was successfully tested with experimental data in \cite{VoKlibanovNguyen:IP2020, Khoaelal:IPSE2021, KlibanovLeNguyenIPI2021, KlibanovKolNguyen:SISC2019} for the inverse scattering problem in the frequency domain given only back scattering data.
Although effective, the convexification method has a drawback. 
It is time consuming. 
In contrast, due to the use of the contraction principle, our new method can quickly deliver a reliable solution to Problem \ref{p}. 
By ``quickly", we mean that the convergence of our method is $O(\theta^k)$ for some $\theta \in (0, 1)$ and $k$ is the number of iteration of our algorithm.

%Results on numerical reconstructions for inverse problems for nonlinear PDEs are rare. 
%We refer to~\cite{Harju2015} 
%for a result on a linearized inverse scattering problem for nonlinear Schr\"odinger equations. A more recent work  
%on the numerical solution to a (linear) inverse source problem for nonlinear parabolic equations can be found in~\cite{LeNguyen:2020}. 
Our method in this paper to solve   a (nonlinear) coefficient inverse problem
for nonlinear hyperbolic equations can be considered as a serious extension of the numerical method  in~\cite{LeNguyen:2020}, which studied an inverse source problem for parabolic equations. 
 Unlike the   convergence analysis  in~\cite{LeNguyen:2020} which was studied for noiseless data we 
consider in this paper the case when the data for the inverse problem has  noise in it.  
We prove that the stability of our method with respect to noise is Lipschitz. Moreover, the nonlinearity considered in this present paper
is also more general than that of the paper~\cite{LeNguyen:2020}. More precisely, the convergence analysis in~\cite{LeNguyen:2020} requires that the nonlinearity term in the parabolic equation  has to be independent of the partial derivatives of 
the  function $u$. Therefore, the extended study in  our paper involves  more  technical difficulties in the convergence analysis of 
the numerical method. 
We also refer the reader to the  publications \cite{Baudouin2013, BAUDOUIN:SIAMNumAna:2017, Baudouin:SIAM2021} for similar approaches. 

In our numerical method   we first derive an approximate Cauchy problem 
for the inverse problem~\ref{p}. The  derivation consists of eliminating the coefficient $c(\x)$ from 
the hyperbolic equation and  approximately transforming the resulting equation into the frequency domain using truncated Fourier series with respect to a special basis,  first introduced in~\cite{Klibanov:jiip2017}.
After the first step, we obtain a system of quasi-linear elliptic equations for the vector $U$ consisting of Fourier coefficients.
The main aim of the second step is to solve this system. 
Writing this system as the form $U = \Phi(U)$ for some map $\Phi$, we recursively define a sequence $\{U_k\}_{k \geq 0}$ as $U_{k+1} = \Phi(U_k)$ where the initial term $U_0$ can be efficiently computed.
The rigorous proof of the convergence of the sequence $\{U_k\}_{k \geq 0}$ to the fixed-point of $\Phi$ is an important strength of this paper. This can be done due to the presence of a Carleman weight function in the definition of $\Phi$. 
Besides the analysis proof, we prove the efficiency of the fixed-point method above by several interesting numerical results.
%
%\textcolor{red}{aaaaa}
%
%In the second step of the numerical method we construct and solve a sequence of linearized versions of the Cauchy 
%problem derived from the first step. A quasi-reversibility is exploited to solve these linearized versions of the Cauchy problem. 
%Using the contraction principle and a Carleman estimate we prove that the solutions  to the linearized versions of the Cauchy problem  
%converge to a function that can be used to compute an approximate solution the inverse problem.  

%We also refer to
%\cite{Baudouin2013, BAUDOUIN:SIAMNumAna:2017, Baudouin:SIAM2021, KlibanovLeNguyenIPI2021, LeKlibanov:preprint2021, KlibanovLiZhang2021, SmirnovKlibanovNguyen:IPI2020, Smirnov:ip2020} for some related results on numerical reconstructions for inverse problems for linear PDEs. The methods studied in this paper including the convexification method  also use Carleman estimates to investigate their convergence analysis.  
%

The paper is organized as follows. In Section~\ref{sec2} we derive an approximate Cauchy problem for Problem \ref{p}. 
Section~\ref{sec3} is dedicated to the construction of the computational algorithm and the convergence analysis of the  method. 
We present numerical examples to illustrate the efficiency of the method in Section~\ref{secNum}. 
Section \ref{sec_concluding} is for some concluding remarks.

\section{An approximate Cauchy problem for Problem \ref{p}} 
\label{sec2}
In this section we derive an approximate Cauchy problem for the inverse problem~\ref{p}. This Cauchy problem serves as an important part for the numerical method  studied in the next section. 
The idea of the derivation  is to eliminate the potential $c(\x)$ from the nonlinear hyperbolic problem~\eqref{waveeqn} and then approximate the resulting equation using a special Fourier basis. 

At the time $t = 0$, we read the hyperbolic equation in \eqref{waveeqn} as
\begin{equation}
	u_{tt}(\x, 0) = \Delta u(\x, 0) + c(\x) F(\x, u(\x, 0), u_t(\x, 0), \nabla u(\x, 0)) = 0	
	\label{2.1}
\end{equation}
for all $\x \in \R^d.$
 Recall from the statement of Problem \ref{p} that $F(\x, p(\x), 0, \nabla p(\x)) \not = 0$ for all $\x \in \overline \Omega.$
 Since $u(\x, 0) = p$ and $u_t(\x, 0) = 0$, it follows from \eqref{2.1} that
 \begin{equation}
 	c(\x) =  \frac{u_{tt}(\x, 0) - \Delta p(\x)}{F(\x, p(\x), 0, \nabla p(\x))}
	\quad 
	\mbox{for all } \x \in \Omega.
	\label{2.2}
 \end{equation}
 Plugging \eqref{2.2} into the hyperbolic equation in \eqref{waveeqn}, we have
\begin{equation}
	u_{tt}(\x, t) = \Delta u(\x, t) 
	+  \frac{\big(u_{tt}(\x, 0) - \Delta p(\x)\big)F(\x, u, u_t, \nabla u)}{F(\x, p(\x), 0, \nabla p(\x))}  
	\quad \mbox{for all } (\x, t) \in\Omega \times (0, T).
	\label{2.3}
\end{equation}

Computing a function $u$, satisfying \eqref{2.3} and the boundary data \eqref{data}, becomes the main aim of this paper.
In fact, having the function $u$ in hand, we can directly compute the potential $c$ via \eqref{2.2}.
However, solving the differential equation \eqref{2.3} is not trivial due to the presence  the complicated nonlinear term $\frac{\big(u_{tt}(\x, 0) - \Delta p(\x)\big)F(\x, u, u_t, \nabla u)}{F(\x, p(\x), 0, \nabla p(\x))}$ which involves a non local function $u_{tt}(\x, 0)$. 
We propose to solve \eqref{2.3} in an approximation context.
Let $\{\Psi_n\}_{n \geq 1}$ be an orthonormal basis of $L^2[0, T]$.
For each $(\x, t) \in \overline \Omega \times [0, T]$, we approximate the wave function $u(\x, t)$ by truncating its Fourier series with respect to $\{\Psi_n\}_{n \geq 1}$ as
\begin{equation}
	u(\x, t) = \sum_{n = 1}^\infty u_n(\x) \Psi_n(t)
	\simeq \sum_{n = 1}^N u_n(\x) \Psi_n(t)
	\label{2.4}
\end{equation}
where
\begin{equation}
	u_n(\x) = \int_0^T u(\x, t) \Psi_n(t)dt \quad n = 1, 2, \dots, N.
	\label{2.5}
\end{equation}
The choice of the cut-off number $N$ will be presented later in Section \ref{secNum}.
Plugging the approximation \eqref{2.4} into \eqref{2.3}, we have
\begin{multline}
	\sum_{n = 1}^N u_n(\x) \Psi_n''(t) 
	\simeq \sum_{n = 1}^N \Delta  u_n(\x) \Psi_n(t)
	+  
	\frac{\sum_{n = 1}^N u_n(\x) \Psi_n''(0) - \Delta p(\x)}{F(\x, p(\x), 0, \nabla p(\x))}
	\\	
	\times	F\Big(\x, \sum_{n = 1}^N u_n(\x) \Psi_n(t), \sum_{n = 1}^N u_n(\x) \Psi_n'(t),  \sum_{n = 1}^N \nabla u_n(\x) \Psi_n(t)\Big)	
	\label{2.6}
\end{multline}
for all $(\x, t) \in \Omega \times (0, T).$ 

\begin{remark}
From now on, we assume that \eqref{2.6} is valid with a suitable choice of $N$, presented later in Section \ref{secNum}.
The analysis to prove \eqref{2.6} as $N \to \infty$ is extremely challenging. 
It is out of the scope of this paper.
Although the rigorous study of the asymptotic behavior of \eqref{2.6} as $N$ large is missing, we do not experience any difficulty in our numerical study.
We refer the readers to \cite{KhoaKlibanovLoc:SIAMImaging2020, LeNguyen:2020, LeNguyenNguyenPowell:JOSC2021, KlibanovLeNguyen:SIAM2020, Nguyen:CAMWA2020, TruongNguyenKlibanov:arxiv2020, Nguyens:jiip2020} for the successful use of similar approximations when the basis $\{\Psi_n\} $ is given in \cite{Klibanov:jiip2017}.
\label{rem2.1}
\end{remark}

From now on, we replace the approximation $``\simeq"$ in \eqref{2.6} by the equality $``="$. 
For each $m \in \{1, 2, \dots, N\}$, multiply $\Psi_m(t)$ to both sides of \eqref{2.6} and then integrate the resulting equation. We obtain
\begin{equation}
	\sum_{n = 1}^N s_{mn} u_n(\x) = \Delta u_m(\x)+ \mathcal F(\x, u_1, u_2, \dots, u_N, \nabla u_1, \cdots, \nabla u_N) 
	\label{2.7}.
\end{equation}
for all $\x \in \Omega, m \in \{1, 2, \dots, N\}$
where 
\[
	s_{mn} = \int_0^T \Psi_n''(t) \Psi_m(t) dt
	\quad 1 \leq m, n \leq N,
\]
 and
\begin{multline*}
	\mathcal F(\x, u_1, u_2, \dots, u_N, \nabla u_1, \cdots, \nabla u_N) 
	=
	\int_0^T \frac{\sum_{n = 1}^N u_n(\x) \Psi_n''(0) - \Delta p(\x)}{F(\x, p(\x), 0, \nabla p(\x))}
	\\
	\times 
	F\Big(\x, \sum_{n = 1}^N u_n(\x) \Psi_n(t), \sum_{n = 1}^N u_n(\x) \Psi_n'(t),  \sum_{n = 1}^N \nabla u_n(\x) \Psi_n(t)\Big)\Psi_m(t)dt.
\end{multline*}
Denote by $U$  the vector $(u_1, \dots, u_N)^{\rm T}$ and $S$ the matrix $(s_{mn})_{n = 1}^N$. We rewrite \eqref{2.7} as 
\begin{equation}
	\Delta U(\x) - SU(\x) + \mathcal F(\x, U(\x), \nabla U(\x)) = 0
	\quad
	\mbox{for all } \x \in \Omega.
	\label{2.8}
\end{equation}

By \eqref{data} and \eqref{2.5}, the Dirichlet and Neumann boundary conditions for the vector $U$ is given by
\begin{equation}
	\left\{
		\begin{array}{ll}
			U(\x) = {\bf f}(\x) := \Big(\ds\int_{0}^T f(\x, t) \Psi_n(t)dt\Big)_{n = 1}^N,\\
			\partial_{\nu} U(\x) = {\bf g}(\x) :=  \Big(\ds\int_{0}^T g(\x, t) \Psi_n(t)dt\Big)_{n = 1}^N,
		\end{array}
	\right. 
	\quad \mbox{for all } \x \in \partial \Omega.
	\label{2.9}
\end{equation}

In order to compute an approximation of a solution to \eqref{2.3} satisfying \eqref{data}, we solve \eqref{2.8}--\eqref{2.9}. 
Then, we compute $u(\x, t)$ via \eqref{2.4}.
We solve \eqref{2.8}--\eqref{2.9} by a new Carleman-based contraction principle.

%\begin{remark}
In practice, the data $f$ and $g$ are measured.
These two functions might contain noise. 
Hence, the boundary data ${\bf f}$ and ${\bf g}$ for the vector valued function $U$ in \eqref{2.9} are noisy, too.
In this case, in order to study problem \eqref{2.8}--\eqref{2.9}, we have to assume that the set of admissible solutions
\[
	H = \big\{
		V \in H^s(\Omega)^N: V|_{\partial \Omega} = {\bf f}, 
		\partial_{\nu} V|_{\partial \Omega} = \mathbf{g}
	\big\}
\]
is nonempty.
Let $f^*$, ${\bf f}^*$, $g^*$ and ${\bf g}^*$ be the noiseless versions of $f$, ${\bf f}$, $g$ and ${\bf g}$ respectively.
Since  ${\bf f}^*$ and ${\bf g^*}$ contain no noise, without loss of the generality, we can assume that 
\begin{equation}
\left\{
	\begin{array}{ll}
		\Delta U^*(\x) - SU^*(\x) + \mathcal F(\x, U^*(\x), \nabla U^*(\x)) = 0
	&
	 \text{ in }  \Omega,\\
	U^* = {\bf f}^* & \text{ on } \partial \Omega,\\
	\partial_{\nu} U^* = {\bf g}^* &  \text{ on } \partial \Omega
	\end{array}
\right.
\label{true}
\end{equation}
has a unique solution $U^*$. 
Since $H$ is nonempty, so is the set
\[
	E = \{ \mathbf{e} \in H^s(\Omega)^N,  \mathbf{e}|_{\partial \Omega} = {\bf f} - {\bf f}^*, \partial_{\nu}  \mathbf{e}|_{\partial \Omega} = {\bf g} - {\bf g}^*\}.
\] 
Let $\delta > 0$ be the noise level. By noise level, we assume that
\begin{equation}
	\inf \{\|\mathbf e\|_{H^s(\Omega)^N}: {\bf e} \in E\} < \delta.
	\label{noise}
\end{equation}
 Assumption \eqref{noise} implies that there exists an ``error" vector valued function ${\bf e} $ satisfying
\begin{equation}
	\left\{
		\begin{array}{ll}
			{\bf e}|_{\partial \Omega} = {\bf f} - {\bf f}^*,\\
			\partial_{\nu} {\bf e}|_{\partial \Omega} = {\bf g} - {\bf g}^*,\\
			\|{\bf e}\|_{H^s(\Omega)^N} < 2\delta.
		\end{array}
	\right.
	\label{errorfunction}
\end{equation}

\begin{remark}
	Since $ {\bf f} - {\bf f}^*$ and ${\bf g} - {\bf g}^*$ are the traces of a vector valued function ${\bf e} \in H^s(\Omega)^N$ as in \eqref{errorfunction}, the noisy data is assumed to be smooth.
 This condition is needed only for the proof of the convergence theorem (see Theorem \ref{thm1}). Our numerical study in Section~\ref{secNum} indicates that the numerical method is able to provide reasonable reconstruction results for nonsmooth noisy data.
 
%  In our numerical study  we compute the noisy data ${\bf f}$ and ${\bf g}$ using formulas \eqref{2.9} with $f$ and $g$ replaced by $f^*(1 + \delta \|f^*\|_{L^2(\partial \Omega)} \mbox{rand})$ and $g^*(1 + \delta \|g^*\|_{L^2(\partial \Omega)} \mbox{rand})$ respectively
%    where $\mbox{rand}$ is a function taking uniformly distributed random numbers in $[-1, 1]$ and $f^*$ and $g^*$ are the noiseless versions of $f$ and $g$ respectively.
    % In Section \ref{s5}, the noise level $\delta = 5\%.$
    \label{remnoise}
\end{remark}

\begin{remark}
	Since problem  \eqref{2.8}--\eqref{2.9} is  an overdetermined problem with noisy Cauchy data, it  might not have a solution. However, our method delivers a  function that  well approximates $U^*.$
%	See Theorem \ref{thm1}.
\end{remark}

\section{The Carleman-based contraction principle}
\label{sec3}
In this section we present the computational algorithm for solving the inverse problem~\ref{p} and the convergence analysis of the algorithm. 
%Let $\x_0$ be a point in $\R^d \setminus \Omega$ and $b > \max_{\x \in \overline \Omega} \{|\x - \x_0|\}.$
%For $\lambda > 1$ and $\beta > 1$, define the Carleman weight function
%\begin{equation}
%	W_{\lambda, \beta}(\x) = e^{\lambda r^\beta(\x)} 
%	\quad \mbox{where }\, 
%	r(\x) = \frac{|\x - \x_0|}{b}.
%	\label{Carl w}
%\end{equation} for all $\x \in \Omega.$ 
%This Carleman weight function plays an important role in our method.
Recall that at the end of Section \ref{sec2}, we reduced Problem \ref{p} to the problem of computing a vector valued function $U$ satisfying \eqref{2.8}--\eqref{2.9}. 
Let $s > \ceil{d/2} + 2$ where $\ceil{d/2}$ is the smallest integer that is greater than or equal to $d/2$. 
We have $H^s(\Omega)$ is continuously embedded into $C^2(\overline \Omega).$ 
Assume that the set of admissible solutions 
\begin{multline*}
	H = \Big\{
		U \in H^s(\Omega)^N: 
		U(\x) = \Big(\ds\int_{0}^T f(\x, t) \Psi_n(t)dt\Big)_{n = 1}^N,
		\\
		\mbox{and }
		\partial_{\nu} U(\x) =  \Big(\ds\int_{0}^T g(\x, t) \Psi_n(t)dt\Big)_{n = 1}^N
		\, \mbox{for all } \x \in \partial \Omega
	\Big\}
\end{multline*}
is nonempty.
Let $\x_0$ be a point in $\R^d \setminus \Omega$. 
Define $r(\x) = |\x - \x_0|/b$ where $b > \max_{\x \in \overline \Omega} \{|\x - \x_0|\}$. 
Fix a regularization parameter $\varepsilon > 0.$
For each $\lambda > 1$ and $\beta > 1$,  define
the functional
\begin{equation}
	J_{\lambda, \beta}(V)(\varphi) = \int_{\Omega}e^{2\lambda r^\beta(\x)} |\Delta \varphi - S \varphi + \mathcal F(\x, V, \nabla V)|^2 d\x 
		+ \varepsilon \|V\|_{H^s(\Omega)^N}^2.
	\label{J}
\end{equation}
Let
$
	\Phi_{\lambda, \beta}: H \to H
$ be defined 
as follows
\[
	\Phi_{\lambda, \beta}(V) = \underset{\varphi \in H}{\mbox{argmin}}	
	J_{\lambda, \beta}(V)(\varphi)
\]
The map $\Phi_{\lambda, \beta}$ is well-defined. 
In fact,
the existence of the minimizer of $J_{\lambda, \beta}$ in $H$ can be proved by the standard arguments in analysis.
Since $H^s(\Omega)^N$ is compactly embedded into $H^2(\Omega)^N$, $J_{\lambda, \beta}$ is  weakly lower semi continuous in $H$. Due the the presence of the regularization term $\varepsilon \|V\|_{H^s(\Omega)^N}^2$, $J_{\lambda, \beta}$ is coercive. 
Hence, $J_{\lambda, \beta}$ has a minimizer in the close and convex set $H$. 
The minimizer  is unique because $J_{\lambda, \beta}$ is strictly convex.

Motivated by the contraction principle, we recursively construct sequence $\{U_k\}_{k \geq 1}$ as in Algorithm \ref{alg}. 
%The convergence of 
%Thank to the presence of the weight function $e^{2\lambda r^\beta(\x)}$  in the integral of the right hand of \eqref{J}, we will prove that the sequence $\{U_n\}_{n \geq 1}$ converges to the true solution $U^*$ to \eqref{true} at the exponential rate.
\begin{algorithm}
	\caption{\label{alg} A computational algorithm for solving Problem \ref{p}}
	\begin{algorithmic}[1]
	\STATE \label{step0} Take any vector valued function $U_0 \in H$.
	\STATE \label{step1}Assume by induction that $U_{k - 1}$ is known. Define $U_k = \Phi_{\lambda, \beta}(U_{k - 1})$.
%	\FOR {$k = 1$ to $k^*$}
%		\STATE Define $U_k = \Phi_{\lambda, \beta}(U_{k - 1})$
%	\ENDFOR
	\STATE Set $U_{\rm comp} = U_{k^*} = (u^{\rm comp}_1, \dots, u^{\rm comp}_{N})$ for some $k^*$ sufficiently large and 
	\begin{equation}
		u^{\rm comp}(\x, t) = \sum_{n = 1}^N u^{\rm comp}_n(\x) \Psi_n(t) 
		\quad \mbox{for all} (\x, t) \in \Omega \times (0, T).
		\label{ucomp}
	\end{equation}
	\STATE Due to \eqref{2.2} and \eqref{ucomp}, set
	\begin{equation}
 	c^{\rm comp}(\x) =  \frac{\sum_{n = 1}^N u^{\rm comp}_n(\x) \Psi_n''(0)  - \Delta p(\x)}{F(\x, p(\x), 0, \nabla p(\x))}
	\quad 
	\mbox{for all } \x \in \Omega.
	\label{ccomp}
 \end{equation}
 as the computed solution to Problem \ref{p}
 	\end{algorithmic}
\end{algorithm}

The following Carleman estimate plays important role to prove the convergence of the sequence $\{U_k\}_{k \geq 1}$.
We refer the reader to Theorem 3.1 and Corollary 3.8  in \cite{LeNguyen:2020} for its proof.
\begin{Lemma}[Carleman estimate]
	There exists a positive constant $\beta_0$ depending only on $b$, $\x_0$, $\Omega$ and $d$ such that for all function $h \in C^2(\overline \Omega)$ satisfying
	\begin{equation}
		h(\x) = \partial_{\nu} h(\x) = 0 \quad \mbox{for all } 
		\x \in \partial \Omega,
		\label{3.1}
	\end{equation}
	the following estimate holds true
	\begin{equation}
		\int_{\Omega} e^{2\lambda r^\beta(\x)}|\Delta h(\x)|^2\,d\x
		\geq
		 C\lambda  \int_{\Omega}  e^{2\lambda r^\beta(\x)}|\nabla h(\x)|^2\,d\x
		+ C\lambda^3  \int_{\Omega}   e^{2\lambda r^\beta(\x)}(\x)|h(\x)|^2\,d\x	
		\label{Car est}	
	\end{equation}
	for all $\beta \geq \beta_0$ and $\lambda \geq \lambda_0$. Here, $\lambda_0 = \lambda_0(b, \Omega, d, \x_0, \beta) > 1$ and $C = C(b, \Omega, d, \x_0, \beta) > 1$ are constants depending only on the listed parameters.
	\label{lemCar}
\end{Lemma}

\subsection{The case when $\|\mathcal F\|_{C^1}$ is  finite}
If $\|\mathcal F\|_{C^1}$ is  finite, we  have the following important theorem, which is the key ingredient for the convergence of the numerical method. The case when $\|\mathcal F\|_{C^1} = \infty$ can be studied by using a cut-off function, see Subsection \ref{Finfty}.
\begin{Theorem}
	Assume that $\|\mathcal F\|_{C^1}$ is finite.
	Fix $\beta = \beta_0$ where $\beta_0$ is the number in Lemma \ref{lemCar}.
	Let $\lambda_0 = \lambda_0(b, \Omega, d, \x_0, \beta)$ be as in Lemma \ref{lemCar}.
	For each $\lambda > \lambda_0$,
	let $\{U_{k}\}_{k \geq 1}$ be the sequence recursively generated by Step \ref{step0} and Step \ref{step1} of Algorithm \ref{alg}.
	Then,
	\begin{multline}
	   \int_{\Omega}  e^{2\lambda r^\beta(\x)}(|U_{k} - U^*|^ 2 +|\nabla (U_{k} - U^*)|^2)\,d\x
	\leq
	\frac{C}{\lambda}  \int_{\Omega} e^{2\lambda r^\beta(\x)}(|U_{k - 1} - U^*|^2 + |\nabla (U_{k - 1} - U^*)|^2)  d\x
	\\
	+   \left(2+ \frac{C}{\lambda}\right) \int_{\Omega} e^{2\lambda r^\beta(\x)} (|{\bf e}|^2 + |\nabla {\bf e}|^2 + |\Delta {\bf e}|^2) d\x
	+ \frac{C}{\lambda} \varepsilon (\|{\bf e}\|^2_{H^s(\Omega)^N}
	+  \|U^*\|^2_{H^s(\Omega)^N})
	\label{3.4444}
\end{multline}
	where $U^*$ is the true solution to \eqref{true}, ${\bf e}$ is the error function satisfying \eqref{errorfunction} and $C$ is a constant depending only on $N,$ $b$, $\Omega$, $d$, $\x_0$, $\beta$ and $\|\mathcal F\|_{C^1}$.
	In particular, we fix $\lambda > \lambda_0$ such that $\theta = \frac{C}{\lambda} \in (0, 1)$ and $\lambda^3 - C > \lambda$. We have
	\begin{multline}
	   \int_{\Omega}  e^{2\lambda r^\beta(\x)} \left(|U_{k} - U^*|^ 2 +|\nabla (U_{k} - U^*)|^2 \right)\,d\x
	\leq
	\theta^k  \int_{\Omega} e^{2\lambda r^\beta(\x)} \left(|U_{0} - U^*|^2 + |\nabla (U_{0} - U^*)|^2 \right)  d\x
	\\
	+   \frac{3(1-\theta^{k+1})}{1 - \theta} \int_{\Omega} e^{2\lambda r^\beta(\x)} \left(|{\bf e}|^2 + |\nabla {\bf e}|^2 + |\Delta {\bf e}|^2 \right) d\x
	+ \frac{\theta \varepsilon(1 - \theta^k)}{1 - \theta} \left(\|{\bf e}\|^2_{H^s(\Omega)^N}
	+  \|U^*\|^2_{H^s(\Omega)^N} \right).
	\label{contraction}
\end{multline}
	\label{thm1}
\end{Theorem}

\begin{proof}%[Proof of Theorem \ref{thm1}]
	Define 
	\begin{equation*}
	H_0 = \Big\{
		U \in H^s(\Omega)^N: 
		U(\x) = 0 \,
		\mbox{and }
		\partial_{\nu} U(\x) =  0
		\, \mbox{for all } \x \in \partial \Omega
	\Big\}.
\end{equation*}
	Fix $k \geq 1$.
	Since $U_k$ is the minimizer of $J_{\lambda, \beta}(U_{k - 1})$ in $H$,
	by the variational principle, for all $h \in H_0$, we have
	\begin{equation}
		\big\langle e^{2\lambda r^\beta(\x)} [\Delta U_{k} - S U_k + \mathcal F(\x, U_{k - 1}, \nabla U_{k - 1})], \Delta h \big\rangle_{L^2(\Omega)^N} 
		+ \varepsilon \langle U_k, h\rangle_{H^s(\Omega)^N} = 0.
		\label{3.4}
	\end{equation}
	Since $U^*$ is the true solution to \eqref{true}, we have
	\begin{equation}
		\big\langle e^{2\lambda r^\beta(\x)} [\Delta U^* - S U^* + \mathcal F(\x, U^*, \nabla U^*)], \Delta h \big\rangle_{L^2(\Omega)^N} 
		+ \varepsilon \langle U^*, h\rangle_{H^s(\Omega)^N} = \varepsilon \langle U^*, h\rangle_{H^s(\Omega)^N}.
		\label{3.5}
	\end{equation}
Combining \eqref{3.4} and \eqref{3.5}, we obtain
\begin{multline}
		\big\langle e^{2\lambda r^\beta(\x)} [\Delta (U_{k} - U^*) - S (U_k - U^*) + 
		(\mathcal F(\x, U_{k - 1}, \nabla U_{k - 1}) - \mathcal F(\x, U^*, \nabla U^*))], \Delta h \big\rangle_{L^2(\Omega)^N} 
		\\
		+ \varepsilon \langle U_k - U^*, h\rangle_{H^s(\Omega)^N} 
		= -\varepsilon \langle U^*, h\rangle_{H^s(\Omega)^N}.
		\label{3.6}
	\end{multline}
%	Recall that $\delta$ is the noise level in the sense of \eqref{noise}, we can find a vector valued function ${\bf e} \in H^s(\Omega)^N$ such that 
%\[
%	\left\{
%		\begin{array}{ll}
%			{\bf e}|_{\partial \Omega} = {\bf f} - {\bf f}^*,\\
%			\partial_{\nu} {\bf e}|_{\partial \Omega} = {\bf g} - {\bf g}^*,\\
%			\|{\bf e}\|_{H^s(\Omega)^N} < 2\delta.
%		\end{array}
%	\right.
%\]
Recall   function ${\bf e}$  in \eqref{errorfunction}. Considering the test function 
	\begin{equation} 
		h = U_k - U^* - {\bf e} \in H_0
		\label{testfn}
	\end{equation} for \eqref{3.6} we derive
\begin{multline}
		\big\langle e^{2\lambda r^\beta(\x)} [\Delta (h + {\bf e}) - S (h + {\bf e}) + 
		(\mathcal F(\x, U_{k - 1}, \nabla U_{k - 1}) - \mathcal F(\x, U^*, \nabla U^*))], \Delta h \big\rangle_{L^2(\Omega)^N} 
		\\
		+ \varepsilon \langle h + {\bf e}, h\rangle_{H^s(\Omega)^N} 
		= -\varepsilon \langle U^*, h\rangle_{H^s(\Omega)^N}.
		\label{3.7}
	\end{multline}	
It follows from \eqref{3.7} that
\begin{multline}
	\int_{\Omega} e^{2\lambda r^\beta(\x)} \big[
	|\Delta h|^2 + 
	\Delta  {\bf e} \Delta h 
	- Sh \Delta h  - S{\bf e} \Delta h 
		+ (\mathcal F(\x, U_{k - 1}, \nabla U_{k - 1}) - \mathcal F(\x, U^*, \nabla U^*)) \Delta h\big]  d\x
		\\
		+ \varepsilon \langle h + {\bf e}, h\rangle_{H^s(\Omega)^N} 
		= -\varepsilon \langle U^*, h\rangle_{H^s(\Omega)^N}.
	\label{3.8}
\end{multline}
Using the inequality $|ab| \leq 8a^2 + \frac{b^2}{32}$ and the Cauchy-Schwarz inequality we deduce from \eqref{3.8} that 
\begin{multline}
	\int_{\Omega} e^{2\lambda r^\beta(\x)}|\Delta h|^2 d\x
	%+ \varepsilon \|h\|^2_{H^s(\Omega)^N}
	\leq 
	C\int_{\Omega} e^{2\lambda r^\beta(\x)} |Sh|^2 d\x + C\int_{\Omega} e^{2\lambda r^\beta(\x)}(|\mathbf{e}|^2 + |\Delta {\mathbf{e}}|^2) d\x
	\\
	 + C\int_{\Omega} e^{2\lambda r^\beta(\x)} |\mathcal F(\x, U_{k - 1}, \nabla U_{k - 1}) - \mathcal F(\x, U^*, \nabla U^*)|^2 d\x
\\
	+ C \varepsilon (\|{\bf e}\|^2_{H^s(\Omega)^N}
	+  \|U^*\|^2_{H^s(\Omega)^N}).
	\label{3.10}
\end{multline}
Since $\|\mathcal F\|_{C^1} < \infty$, we can find a number $C_{\mathcal F}$ such that
	\begin{equation}
		|\mathcal F(\x, U_{k - 1}, \nabla U_{k - 1}) - \mathcal F(\x, U^*, \nabla U^*)|^2 
		\leq C_{\mathcal F}
		(|U_{k - 1} - U^*|^2 + |\nabla (U_{k - 1} - U^*)|^2) 
	\label{3.11}
	\end{equation}
	for all $\x \in \Omega.$
Fix $\beta \geq \beta_0$ and consider $C$ as a generic constant, depending on $N$, $T$, $\mathcal F$, $b,$ $\Omega,$ $d,$ $\x_0$ and $\beta$.
It implies from \eqref{testfn},  \eqref{3.10} and \eqref{3.11} that
\begin{multline}
	\int_{\Omega} e^{2\lambda r^\beta(\x)}|\Delta h|^2 d\x
	\leq 
	 C \int_{\Omega} e^{2\lambda r^\beta(\x)}|h|^2 d\x 	 + C \int_{\Omega} e^{2\lambda r^\beta(\x)} (|{\bf e}|^2 + |\Delta {\bf e}|^2) d\x
	 \\
	 + C  \int_{\Omega} e^{2\lambda r^\beta(\x)}(|U_{k - 1} - U^*|^2 + |\nabla (U_{k - 1} - U^*)|^2)  d\x
	 \\
	+ C \varepsilon (\|{\bf e}\|^2_{H^s(\Omega)^N}
	+  \|U^*\|^2_{H^s(\Omega)^N}).
	\label{3.12}
\end{multline}
Since $h$ is in $H_0$, the Carleman estimate \eqref{Car est} for $h$ is valid. 
Thanks to \eqref{Car est} and \eqref{3.12} and the fact that $\lambda^3 - C > \lambda$ for $\lambda$ large enough, we have
\begin{multline}
	 \lambda   \int_{\Omega}  e^{2\lambda r^\beta(\x)}(|h|^ 2 +|\nabla h(\x)|^2)\,d\x
	\leq
	C  \int_{\Omega} e^{2\lambda r^\beta(\x)}(|U_{k - 1} - U^*|^2 + |\nabla (U_{k - 1} - U^*)|^2)  d\x
	\\
	+   C \int_{\Omega} e^{2\lambda r^\beta(\x)} (|{\bf e}|^2 + |\Delta {\bf e}|^2) d\x
	+ C \varepsilon (\|{\bf e}\|^2_{H^s(\Omega)^N}
	+  \|U^*\|^2_{H^s(\Omega)^N}).
	\label{3.13}
\end{multline}
Recalling \eqref{testfn} and using the inequality $(a + b)^2 \geq \frac{a^2}{2} - b^2$, we obtain from \eqref{3.13} that
\begin{multline}
	\frac{ \lambda }{2}  \int_{\Omega}  e^{2\lambda r^\beta(\x)}(|U_{k} - U^*|^ 2 +|\nabla (U_{k} - U^*)|^2)\,d\x
	\leq
	C  \int_{\Omega} e^{2\lambda r^\beta(\x)}(|U_{k - 1} - U^*|^2 + |\nabla (U_{k - 1} - U^*)|^2)  d\x
	\\
	+   (C + \lambda) \int_{\Omega} e^{2\lambda r^\beta(\x)} (|{\bf e}|^2 + |\nabla {\bf e}|^2 + |\Delta {\bf e}|^2) d\x
	+ C \varepsilon (\|{\bf e}\|^2_{H^s(\Omega)^N}
	+  \|U^*\|^2_{H^s(\Omega)^N}).
	\label{3.15}
\end{multline}
Dividing both sides of \eqref{3.15} by $\lambda/2$ we obtain  \eqref{3.4444}.

Fix $\lambda > \lambda_0$ such that $\theta = C/\lambda \in (0, 1)$ and $\lambda^3 - C > \lambda$.
Replacing $k$ by $k-1$ in \eqref{3.4444} and multiplying the resulting equation by $\theta$, we have
\begin{multline}
	\theta   \int_{\Omega}  e^{2\lambda r^\beta(\x)}(|U_{k-1} - U^*|^ 2 +|\nabla (U_{k-1} - U^*)(\x)|^2)\,d\x
	\\
	\leq
	\theta^2  \int_{\Omega} e^{2\lambda r^\beta(\x)}(|U_{k - 2} - U^*|^2 + |\nabla (U_{k - 2} - U^*)|^2)  d\x
	+   (2\theta + \theta^2)  \int_{\Omega} e^{2\lambda r^\beta(\x)} (|{\bf e}|^2 + |\nabla {\bf e}|^2 + |\Delta {\bf e}|^2) d\x
	\\
	+ \theta^2 \varepsilon (\|{\bf e}\|^2_{H^s(\Omega)^N}
	+  \|U^*\|^2_{H^s(\Omega)^N}).
	\label{3.16}
\end{multline}
Combining \eqref{3.4444} and \eqref{3.16}, we have
\begin{multline*}
	   \int_{\Omega}  e^{2\lambda r^\beta(\x)}(|U_{k} - U^*|^ 2 +|\nabla (U_{k} - U^*)|^2)\,d\x
	\leq
	\theta^2  \int_{\Omega} e^{2\lambda r^\beta(\x)}(|U_{k - 2} - U^*|^2 + |\nabla (U_{k - 2} - U^*)|^2)  d\x
	\\
	+   (2 + 3\theta + \theta^2) \int_{\Omega} e^{2\lambda r^\beta(\x)} (|{\bf e}|^2 + |\nabla {\bf e}|^2 + |\Delta {\bf e}|^2) d\x
	+ \theta \varepsilon (1 + \theta) (\|{\bf e}\|^2_{H^s(\Omega)^N}
	+  \|U^*\|^2_{H^s(\Omega)^N}).
\end{multline*}
By induction, we obtain \eqref{contraction}.
\end{proof}

\begin{remark}
Define the norm
\[
	\|V\|_{H^1_{\lambda, \beta}(\Omega)^N} 
	= \Big(\int_{\Omega} e^{2\lambda r^\beta(\x)} (|V|^2 + |\nabla V|^2)d\x\Big)^{1/2}
	\quad \mbox{for all } V \in H^1(\Omega)^N.
\]
A direct consequence of Theorem \ref{thm1} and \eqref{contraction} is that if we fix $\beta = \beta_0$ and $\lambda > \lambda_0$ such that $\theta = C/\lambda \in (0, 1)$ and that $\lambda^3 - C > \lambda$, where $C$ is the constant in \eqref{3.4444}, then the sequence $\{U_k\}_{k \geq 1}$ well-approximates the vector valued function $U^*$ with respect to the norm $\|\cdot\|_{H^1_{\lambda, \beta}(\Omega)^N} $. 
The rate of the convergence is $O(\theta^k)$ as $k \to \infty.$
Due to \eqref{errorfunction} and \eqref{contraction}, the error in computation is $O(\delta + \sqrt{\varepsilon}\|U^*\|_{H^s(\Omega)^N})$.
We say that Algorithm is Lipschitz stable with respect to noise.
\label{rem3.1}
\end{remark}

\begin{Corollary}
Fix $\lambda$ and $\beta$ as in Remark \ref{rem3.1} such that $\theta \in (0, 1).$
Define
\begin{equation*}
 	c^*(\x) =  \frac{\sum_{n = 1}^N u^*_n(\x) \Psi_n''(0)  - \Delta p(\x)}{F(\x, p(\x), 0, \nabla p(\x))}
	\quad 
	\mbox{for all } \x \in \Omega.
\end{equation*}
In our approximation context in \eqref{2.6} and Remark \ref{rem2.1}, $c^*$ is the true solution to Problem \ref{p}.
 Then, due to \eqref{contraction} and \eqref{ccomp}, we have
\[
	\|c^k - c^*\|_{H^1_{\lambda, \beta}(\Omega)} \leq 
	C \theta^k \|U_0 - U^*\|_{H^1_{\lambda, \beta}(\Omega)} + C(\delta + \sqrt{\varepsilon}\|U^*\|_{H^s(\Omega)^N})
\]
where $C$ depends only on $N,$ $b$, $\Omega$, $d$, $\x_0$, $\beta$, $\lambda$ and $\|\mathcal F\|_{C^1}$.
\end{Corollary}

\subsection{The case when $\|\mathcal F\|_{C^1} = \infty$} \label{Finfty}

In the case $\|\mathcal F\|_{C^1}$ is infinite, we need  additional information about the true solution $U^*$.
Assume that we know a large number $M$ such that $|U^*(\x)| + |\nabla U^*(\x)| < M$ for all $\x \in \overline \Omega$.
Define a smooth cut off function $\chi: \overline \Omega \times \R^N \times \R^{N \times d} \to \R$ that satisfies
\[
	\chi(\x, \xi, \p) = \left\{
	\begin{array}{ll}
		 1 & |\xi| + |\p| < M,\\
		c \in (0, 1) & M < |\xi| + |\p| < 2M,\\
		0 &\mbox{otherwise},
	\end{array}
	\right.
	\quad \mbox{for all } \x \in \overline \Omega, \xi \in \R^N, \p \in \R^{N \times d}.
\]
Define $\widetilde{\mathcal F} = \chi \mathcal F.$
It is clear that $U^*$ is the solution to 
\begin{equation}
\left\{
	\begin{array}{ll}
		\Delta U^*(\x) - SU^*(\x) + \widetilde{\mathcal F}(\x, U^*(\x), \nabla U^*(\x)) = 0
	&
	 \x \in \Omega,\\
	U^* = {\bf f}^* & \x \in \partial \Omega,\\
	\partial_{\nu} U^* = {\bf g}^* & \x \in \partial \Omega.
	\end{array}
\right.
\label{true1}
\end{equation}
Since $\mathcal F$ is smooth, it is locally bounded in $C^1$. Hence, $\|\widetilde{\mathcal F}\|_{C^1}$ is a finite number. 
We now can apply Algorithm \ref{alg} to solve \eqref{true1} to compute $U^*.$

\section{Numerical study} \label{secNum}

In this section, we present important details of the implementation of   Algorithm \ref{alg} and display some numerical results obtained by the algorithm when the dimension $d = 2$.
In our numerical study, the computational domain $\Omega = (-1, 1)^2$ which is uniformly partitioned into a $64\times 64$ meshgrid. The final time $T = 2$. The initial condition $p(\x) = 0.5$ for all $\x$. The number of terms in the truncated Fourier series with respect to the basis $\{\Psi_n\}_{n\geq 1}$ is $20$ for all the involved functions.

\subsection{Data generation}

To generate data for the inverse problem, we first approximate the governing initial value problem \eqref{waveeqn} by 
\begin{equation}
	\left\{
		\begin{array}{rcll}
			u_{tt}(\x, t) &=& \Delta u(\x, t) + c(\x) F(\x, u, u_t, \nabla u) &(\x, t) \in G \times (0, T),\\
			u(\x, 0) &=& p(\x) &\x \in G,\\
			u_t(\x, 0) &=& 0 &\x \in G.
		\end{array}
	\right.
	\label{approWave}
\end{equation}
where $G$ is the open square $(-R_1, R_1)^2$ with $R_1 = 4$. This choice of $R_1$ ensures that at the final time $t = T$, $u(\x,T) = p(\x)$ for all $\x\in \mathbb{R}^2\setminus G$ and therefore the restriction onto $\overline{\Omega}$ of the solution $u(\x,t)$ to \eqref{approWave} coincides with that of the solution to the original problem \eqref{waveeqn} for $t\leq T$. Problem~\eqref{approWave} is then solved using an explicit finite difference scheme as follows,
$$
\left\{
\begin{array}{l}
u_0 =  u_1 = p, \\
u_{j+2}= \Delta t^2\left[\Delta u_{j+1} + cF\left(\x,u_{j+1},\frac{u_{j+1}-u_j}{\Delta t},\nabla u_{j+1}\right)\right] + 2u_{j+1} - u_j,\quad \x\in G
\end{array}
\right.
$$
where $u_j(\x) = u(\x,t_j)$, the interval $[0,T]$ is uniformly discretized  as $0 = t_0 < t_1 < \dots < t_{256} = T$ with time step size $\Delta t = t_1-t_0$. After getting the solution, we extract $u(\x,t_j)$ and  compute its normal derivative using finite differences for all $j$ and $\x\in \partial \Omega$ to form the data $\bf f^*$ and $\bf g^*$ for the inverse problem. Note that the vectors $\bf f^*$ and $\bf g^*$ are the discretized versions of their counterparts in the theory. We then add $5\%$ of artificial noise to $\bf f^*$ and $\bf g^*$ to obtain the noisy data $\bf f$ and $\bf g$, i.e.
$$
\textbf{f} = \textbf{f}^* + 0.05\|\textbf{f}^*\|_2\text{rand}(-1,1)
$$
where rand$(-1,1)$ is a unit vector consisting of normally distributed random numbers within the interval $(-1,1)$ and $\|\textbf{f}^*\|_2$ is the vector $2$-norm of $\bf f^*$.

\subsection{Numerical examples for $F = \sqrt{|u|} + |\nabla u|$}

We consider $F = \sqrt{|u|} + |\nabla u|$. The test objects are: two disks of the same size, one has potential $2$ and the other has potential $1$; a kite-shaped object with potential $2$; a peanut-shaped object with potential $2$. 

\begin{figure}[h!!!]
\centering
\subfloat[Target]{\includegraphics[width=6cm]{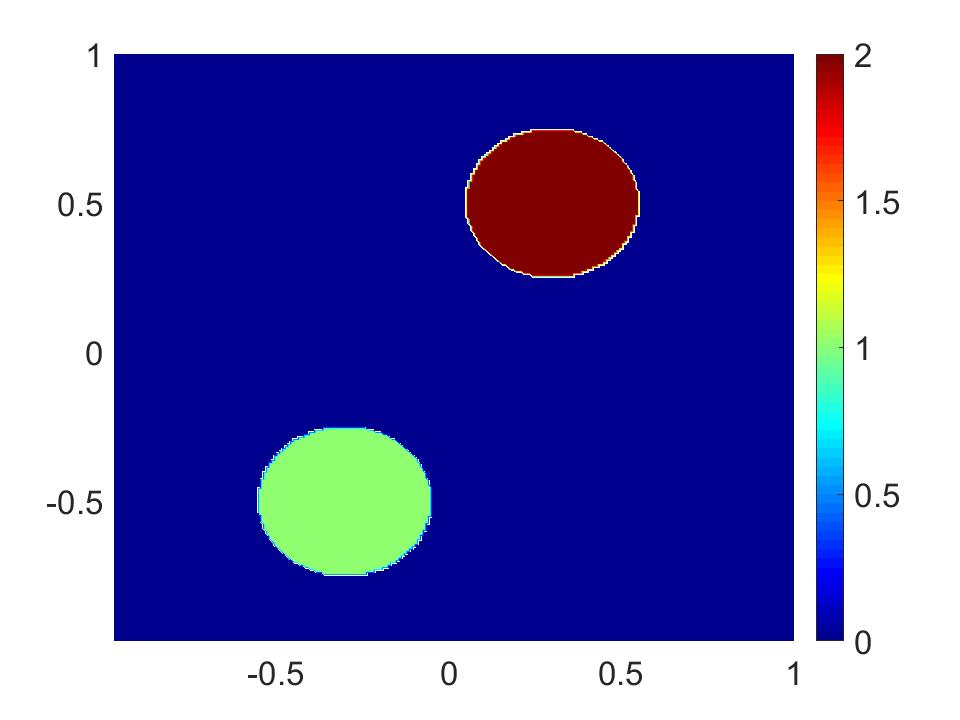}} \hspace{0cm}
\subfloat[Reconstruction at  first iteration]{\includegraphics[width=6cm]{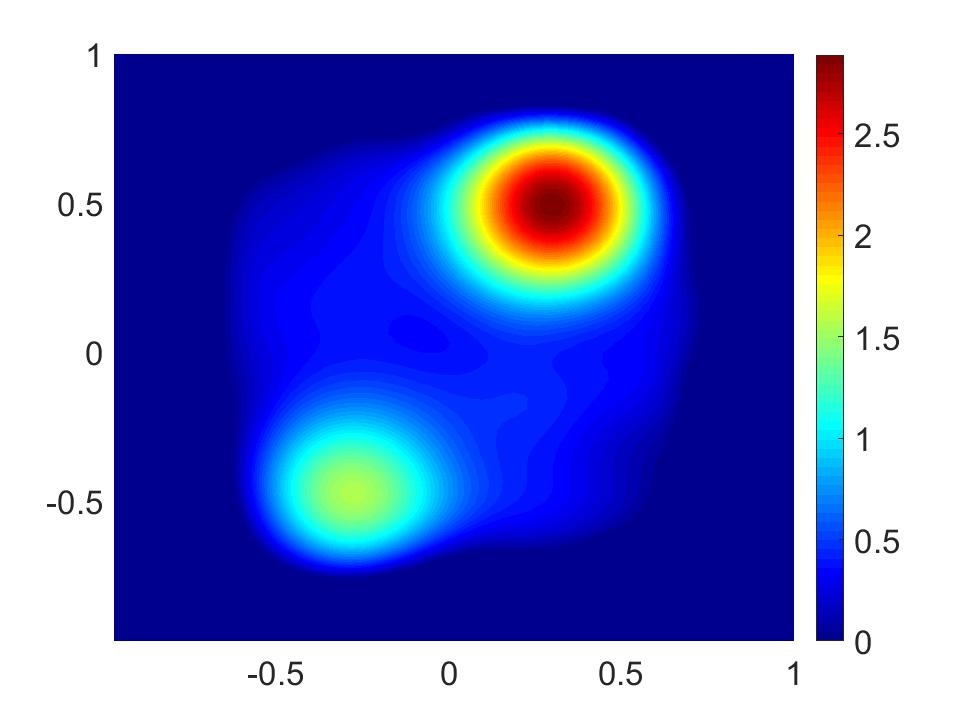}}  \hspace{0cm} 
\subfloat[Reconstruction at 25th iteration]{\includegraphics[width=6cm]{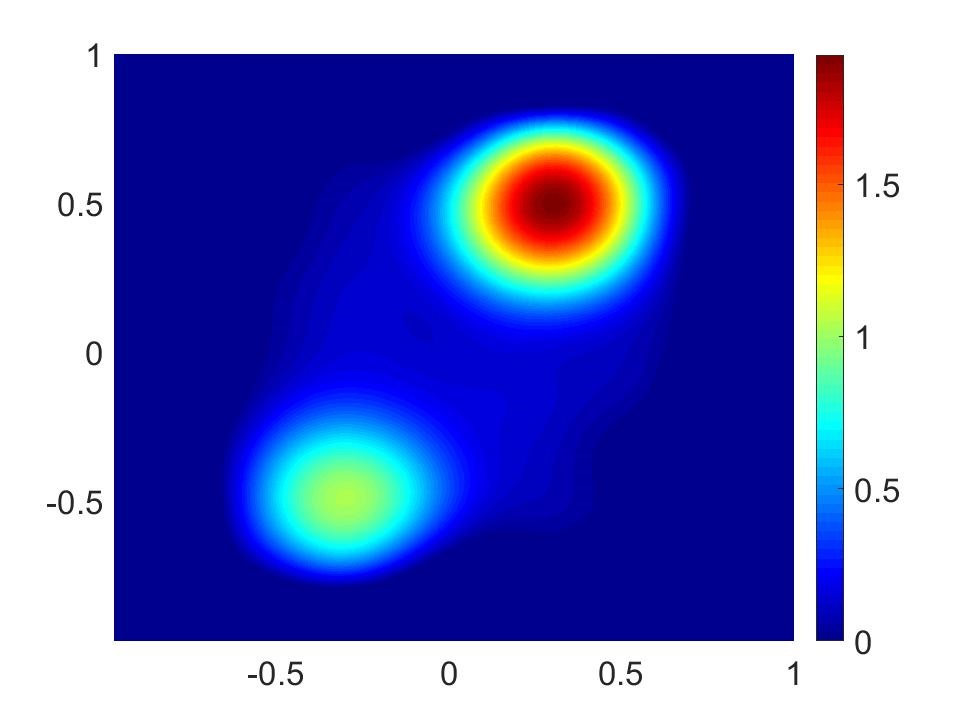}} \hspace{0cm}
\subfloat[Cost function]{\includegraphics[width=6cm]{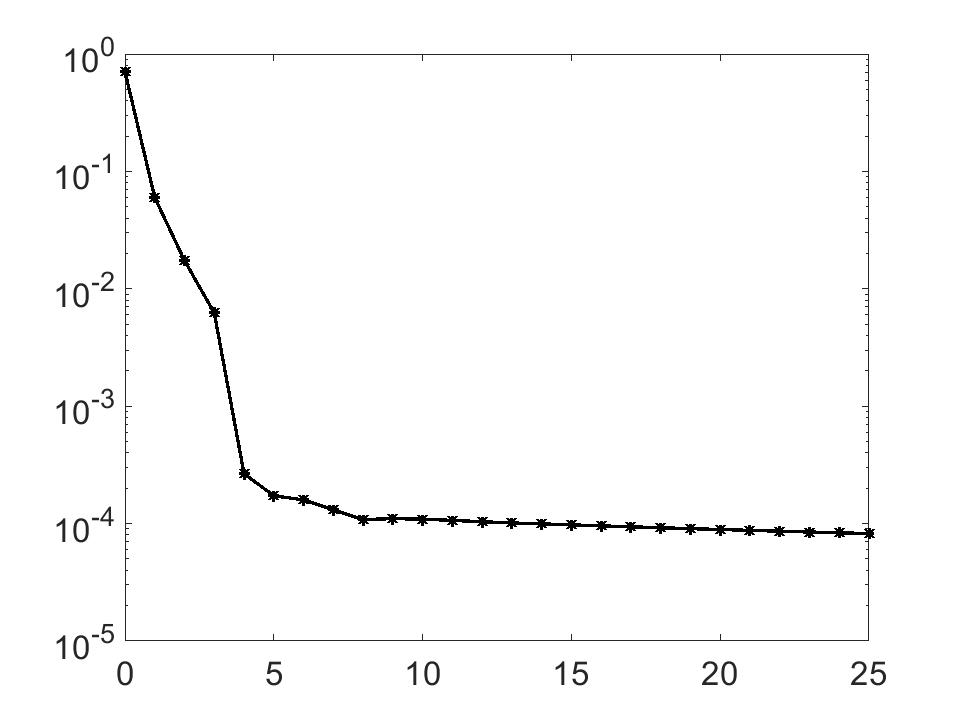}}  \hspace{0cm} 
\caption{Reconstruction of a potential supported in two disks.} 
 \label{fi1}
\end{figure}

\begin{figure}[h!!!]
\centering
\subfloat[Target]{\includegraphics[width=6cm]{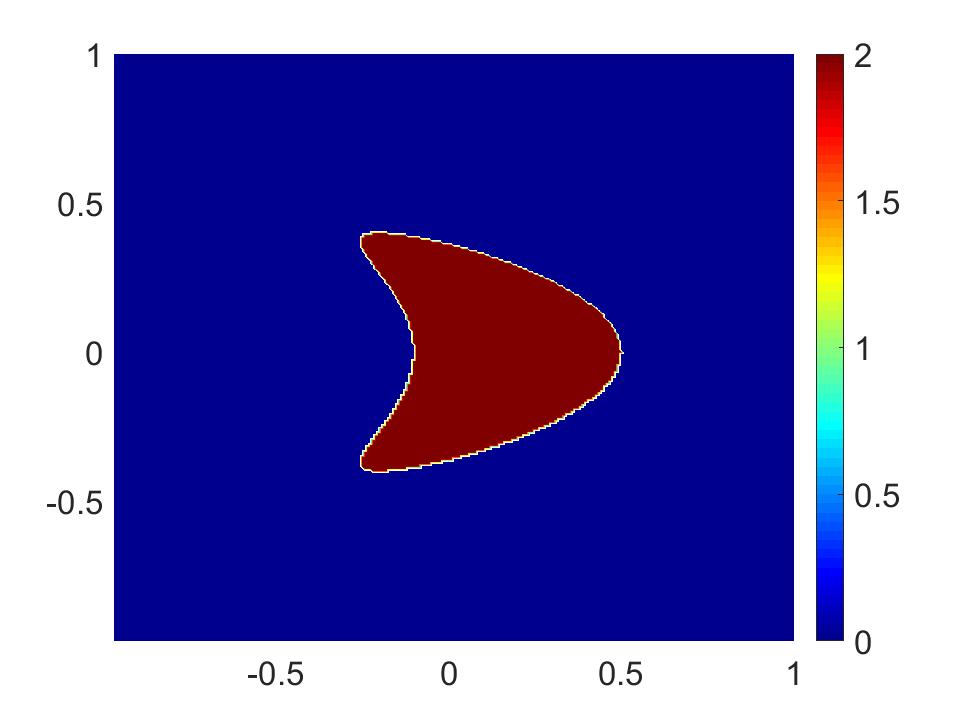}} \hspace{0cm}
\subfloat[Reconstruction at first iteration]{\includegraphics[width=6cm]{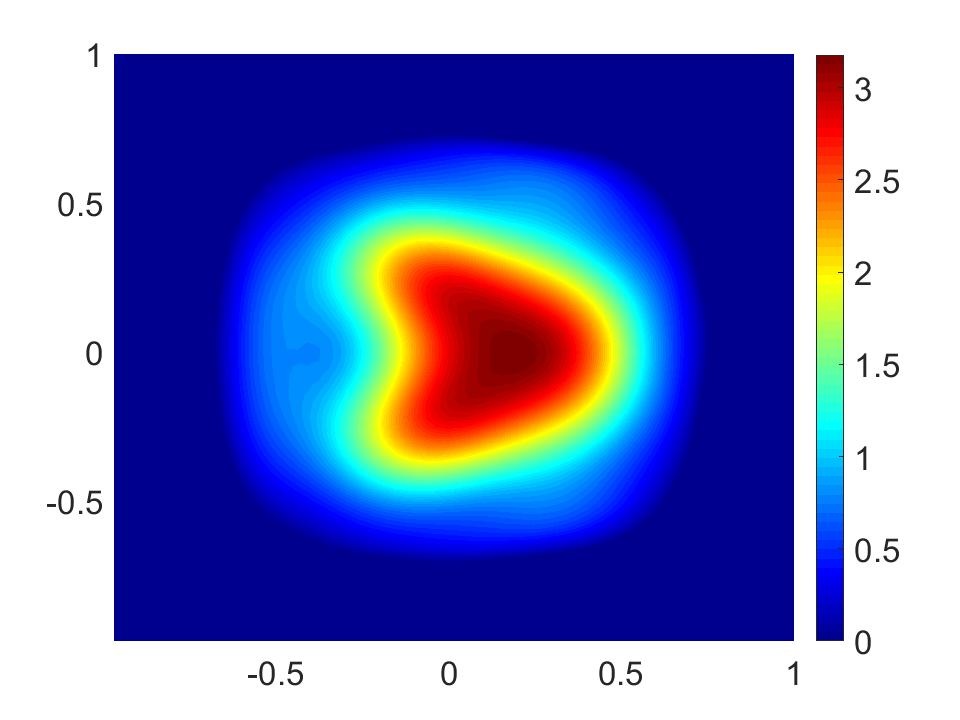}}  \hspace{0cm} 
\subfloat[Reconstruction at 25th iteration]{\includegraphics[width=6cm]{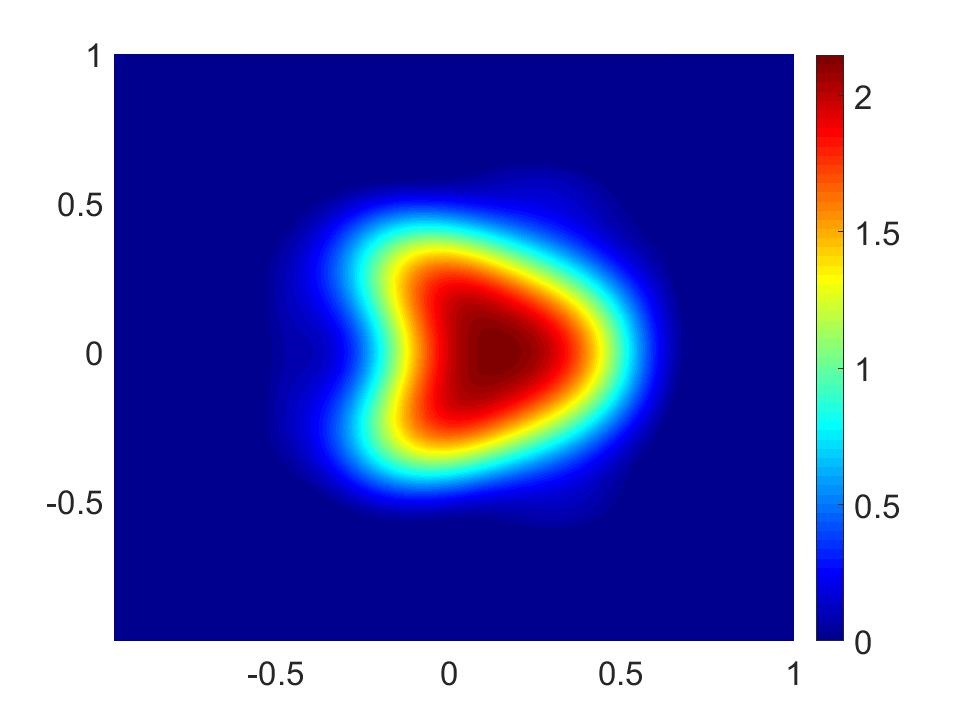}} \hspace{0cm}
\subfloat[Cost function]{\includegraphics[width=6cm]{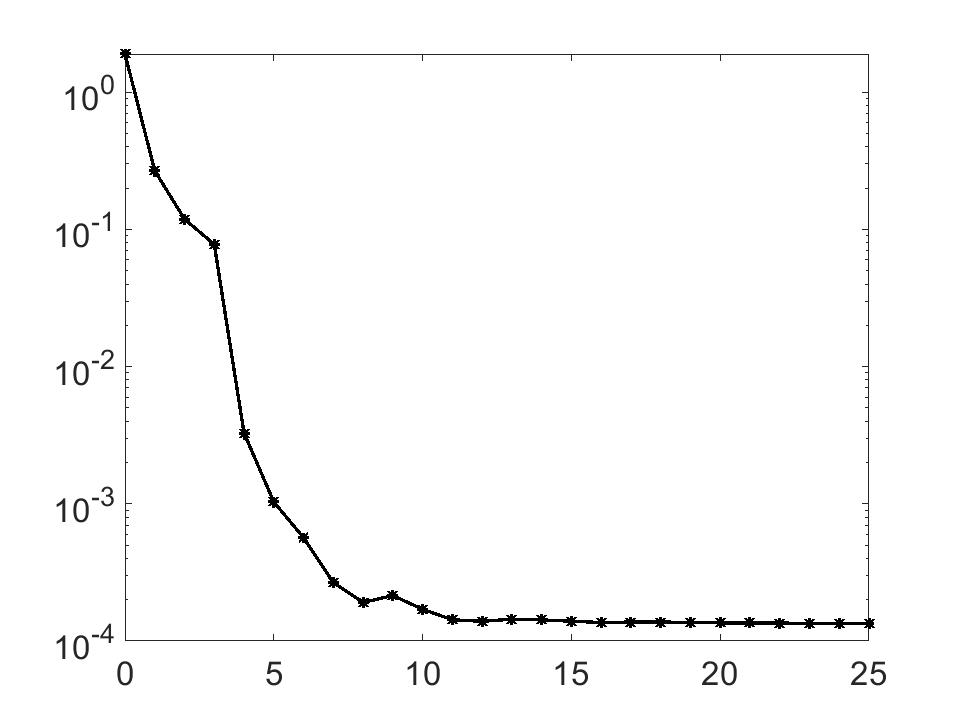}}  \hspace{0cm} 
\caption{Reconstruction of a  potential with kite-shaped support.
 } 
 \label{fi2}
\end{figure}

\begin{figure}[h!!!]
\centering
\subfloat[Target]{\includegraphics[width=6cm]{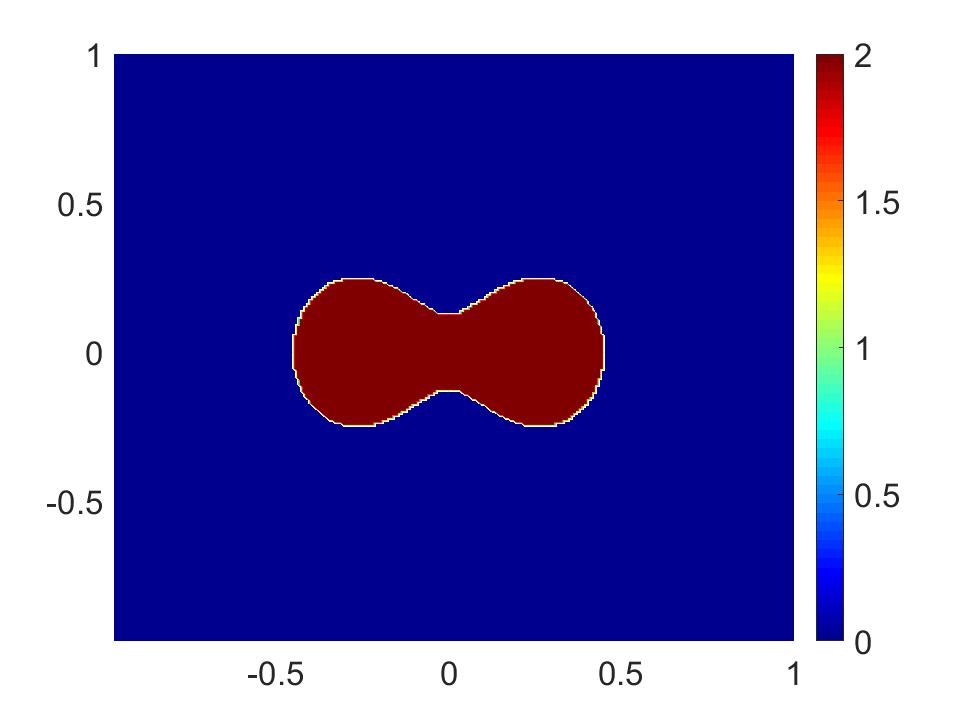}} \hspace{0cm}
\subfloat[Reconstruction at first iteration]{\includegraphics[width=6cm]{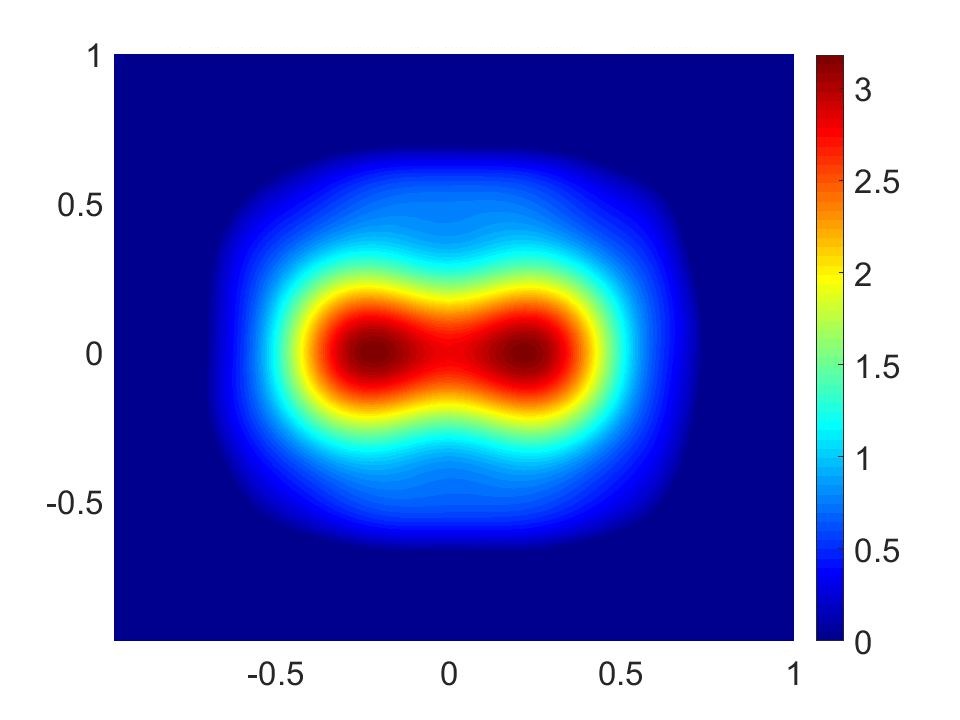}}  \hspace{0cm} 
\subfloat[Reconstruction at 25th iteration]{\includegraphics[width=6cm]{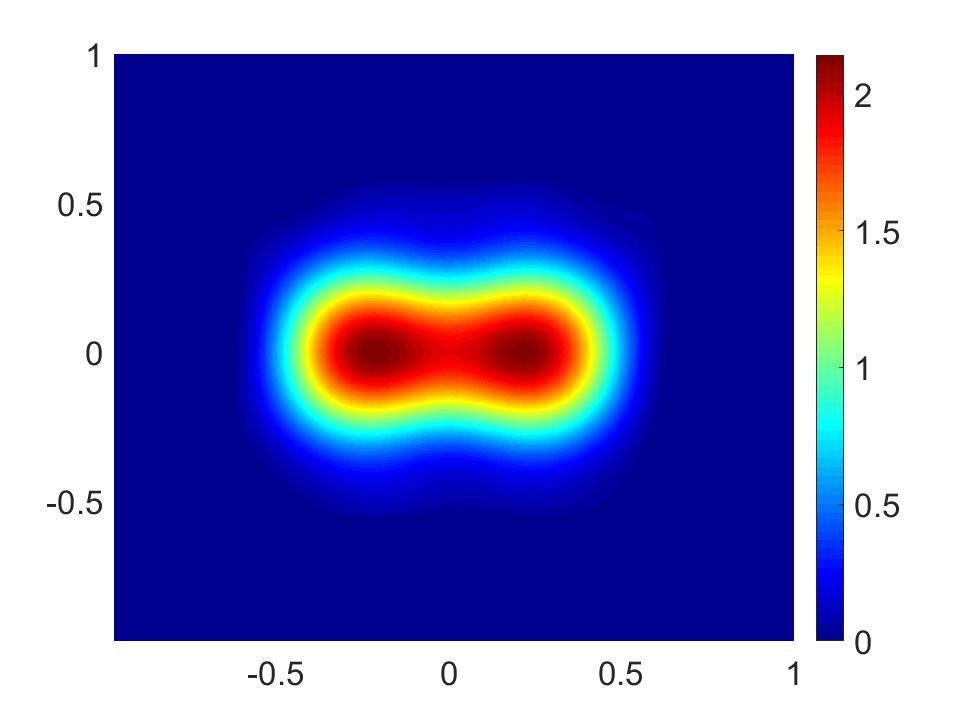}} \hspace{0cm}
\subfloat[Cost function]{\includegraphics[width=6cm]{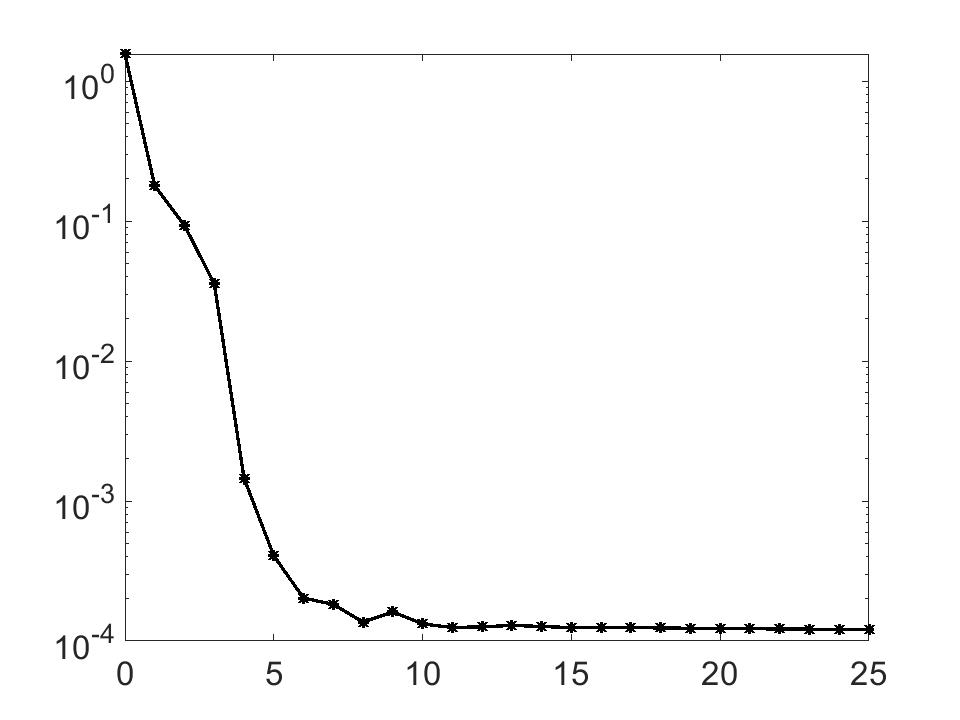}}  \hspace{0cm} 
\caption{Reconstruction of a potential with peanut-shaped support.
 } 
 \label{fi3}
\end{figure}

We choose the regularization parameter $\varepsilon = 7\times 10^{-5}$. Note that in the simulation, we used the $H^2(\Omega)^N$-norm for the regularization term. However, in order to simplify calculations, we only included the $L^2$-norms of the second-order partial derivatives of $U_k$. According to our observation, this simplification did not affect the reconstructions, but helped reduce computation time. 

For the Carleman weight function $e^{2\lambda r^\beta(\x)}$, we choose 
$$\lambda = 2, \quad \beta = 10, \quad r(\x) = \frac{|\x - (0,1.25)|}{3}.$$ 

It is natural to compute the initial solution $U_0$ by solving the problem obtained by removing from \eqref{2.8}--\eqref{2.9} the nonlinear term $\mathcal F(\x, U(\x), \nabla U(\x))$. 
In other words, we compute
the initial term $U_0$  solving for the Tikhonov regularized solution of the linear equation
$$
\left\{
\begin{array}{ll}
\Delta U(\x) - SU(\x) = 0 & \x \in \Omega, \\
U = \bf f & \x\in \partial \Omega,\\
\partial_\nu U = \bf g & \x\in \partial \Omega.
\end{array}
\right.
$$
We experience that this choice of $U_0$ helps reduce the computational time.
%
%	It is natural to compute $U_0$ in Step \ref{step0} of Algorithm \ref{alg} by minimizing the functional obtained by removing from $J_{\lambda, \beta}$ the nonlinear term $\mathcal F(\x, V, \nabla V).$ We experience that this choice of $U_0$ helps reduce the computational time.
%	\label{rem31}

At the $k$th iteration, we compute the cost function
$$
J(U_k) = \int_{\Omega}e^{2\lambda r^\beta(\x)} |\Delta U_k - S U_k + \mathcal F(\x, U_k, \nabla U_k)|^2 d\x + \varepsilon \sum_{|\alpha|=2}\left\|D^\alpha U_k\right\|_{L^2(\Omega)^N}^2
$$
to observe convergence of the method. The discretization of the cost function is done using finite difference approximations. We do not describe the details of the discretization here since a similar discretization with more details can be found in~\cite{LeNguyen:2020}. 
We see that after about $10$ iterations, the value of the cost function starts to stabilize, and therefore we stopped the algorithm after $25$ iterations. The pictures of the true target, the reconstruction at the first and last iteration, along with the values of the cost functions throughout $25$ iterations are shown in Figure \ref{fi1} - \ref{fi3}. From the pictures of the cost functions we can see that the proposed method converges quickly.   These results also indicate that the proposed method is able to provide reasonable and robust reconstructions of both the value and the support of several  potentials.

\subsection{Numerical examples for $F = |u|^2 + |\nabla u|^2$}

The numerical method is also tested  with another function $F$ which has a higher level of nonlinearity ($F = |u|^2 + |\nabla u|^2$). We consider the same test cases and parameters as in the case when $F = \sqrt{|u|} + |\nabla u|$. We obtain similar results for both the behavior of the cost function and the quality of the reconstructions. To avoid a repeat of  several similar images we show only the reconstruction results at the 25th iterations in Figure \ref{fi4}. 

\begin{figure}[h!!!]
\centering
\subfloat[Two disks]{\includegraphics[width=5.7cm]{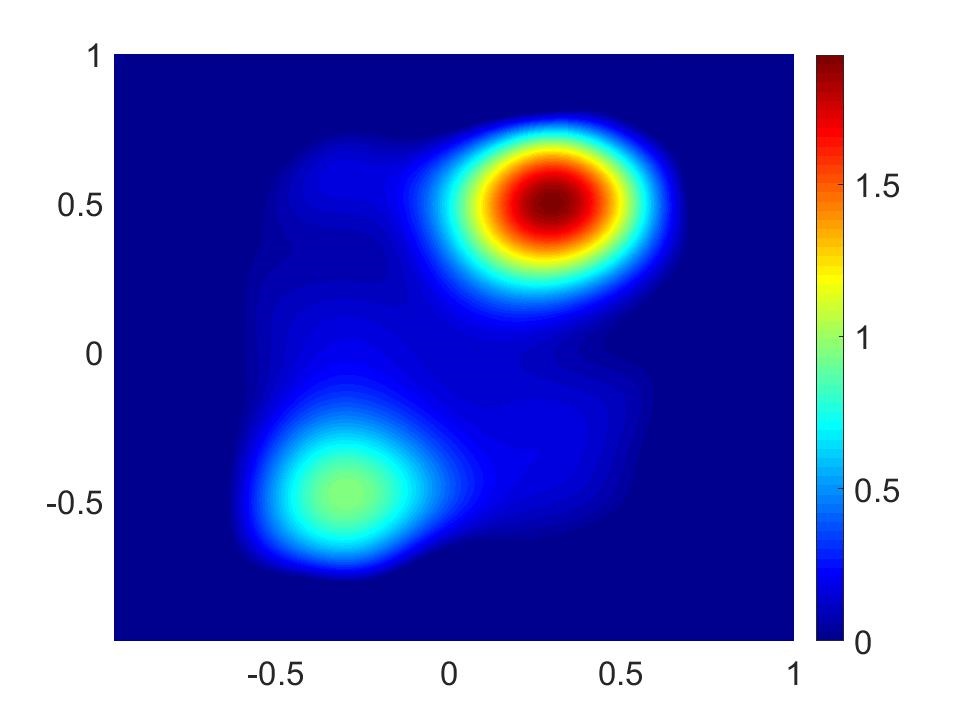}}  \hspace{-0.4cm} 
\subfloat[Kite-shaped object]{\includegraphics[width=5.7cm]{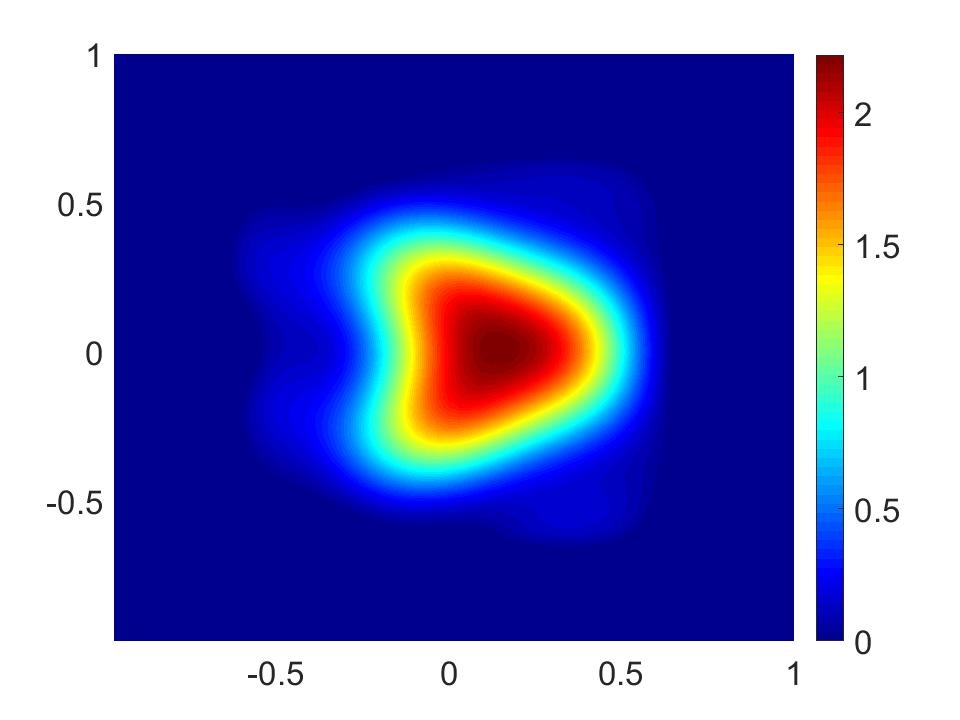}} \hspace{-0.4cm}
\subfloat[Peanut-shaped object]{\includegraphics[width=5.7cm]{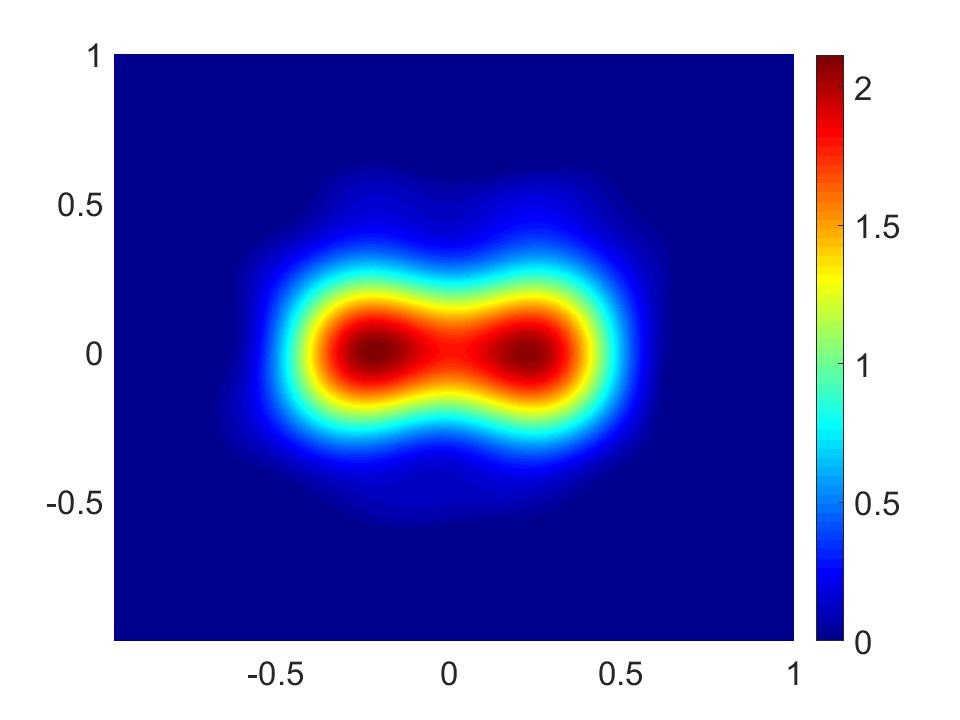}}  \hspace{-0.4cm} 
\caption{Reconstructions with $F = |u|^2 + |\nabla u|^2$.
 } 
 \label{fi4}
\end{figure}

\section{Concluding remarks}\label{sec_concluding}

We have studied a nonlinear inverse problem of reconstructing the potential coefficient of nonlinear hyperbolic equation. 
Although this problem is nonlinear, we solve it globally. That means we do not require a good initial guess. 
Our method consists of two main steps.
In the first step, we reduce this coefficient inverse problem to the problem of solving a system of quasi-linear PDEs with Cauchy boundary data. Solving this system is the main goal of the second step. 
We  define an operator $\Phi$ such that the true solution to the inverse problem is the fixed point of $\Phi$.
This fixed point is computed by constructing a recursive sequence as in the proof of the classical contraction principle.
We next exploit a Carleman estimate to prove the convergence of this sequence.
We obtain  an exponential rate of convergence.
Moreover, we have proved that the stability of our method with respect to noise is of  Lipschitz type.
Numerical results are out of expectation.
 
\section*{Acknowledgement}   
The work of DLN and TT was partially supported by NSF grant DMS-1812693.
The work of LHN was supported by US Army Research Laboratory and US Army Research
Office grant W911NF-19-1-0044 and  was supported, in part, by funds provided by the Faculty
Research Grant program at UNC Charlotte.

\bibliographystyle{plain}
\bibliography{mybib_LiemTrung}
\end{document}